\documentclass[12pt]{article} 

\usepackage[utf8]{inputenc} 

\usepackage{amssymb}
\usepackage{amsfonts}
\usepackage{amsmath}
\usepackage{amsxtra}
\usepackage{amsthm}
\usepackage{mathrsfs}
\usepackage{enumerate}

\usepackage{geometry} 
\usepackage{graphicx} 

\usepackage{booktabs} 
\usepackage{array} 
\usepackage{paralist} 
\usepackage{verbatim} 
\usepackage{subfig} 

\usepackage{fancyhdr} 
\newtheorem{thm}{Theorem}
\newtheorem{lem}[thm]{Lemma}  
\renewcommand{\epsilon}{\varepsilon}

\title{Homogenization of a Boundary Obstacle Problem}
\author{Ray Yang}
\date{} 

\begin{document}
\maketitle
\begin{abstract}
We prove the existence of a homogenization limit for solutions of appropriately formulated sequences of boundary obstacle problems for the Laplacian on $C^{1,\alpha}$ domains. Specifically, we prove that the energy minimizers $u_\epsilon$ of $\int |\nabla u_\epsilon|^2 dx$, subject to $u \geq \phi$ on a subset $S_\epsilon$, converges weakly in $H^1$ to a limit $\bar{u}$ which minimizes the energy $\int |\nabla \bar{u}|^2 dx + \int_\Sigma (u-\varphi)_-^2 \mu(x) dS_x$, $\Sigma \subset \partial D$, if the obstacle set $S_\epsilon$ shrinks in an appropriate way with the scaling parameter $\epsilon$. This is an extension of a result by Caffarelli and Mellet \cite{Caffarelli-Mellet-2}, which in turn was an extension of a result of Cioranescu and Murat \cite{Cioranescu-Murat}. 
\end{abstract}
\section{Introduction}
In this paper we consider the homogenization of a sequence of boundary obstacle problems. The boundary obstacle problem was used by Duvaut and Lions \cite{Duvaut-Lions} to describe a semipermeable membrane, that is to say, a membrane which permits fluid to flow in one direction, but not the other. We can think of the obstacle function $\varphi(x)$ as representing the outside pressure or concentration, and $S_\epsilon$ as the subset of the boundary which is composed of a semipermeable membrane that permits fluid to flow, $S_\epsilon$ being composed of ``patches'' of boundary, of appropriate size and separation given by $\epsilon$. 

It is natural to ask, as the scaling parameter $\epsilon \rightarrow 0$, what the limiting behavior might be. This is the realm of homogenization theory. Caffarelli and Mellet \cite{Caffarelli-Mellet-2} answered this question in the case where $S_\epsilon$ lies along a flat portion of the boundary, and consists of patches contained in small balls about points of a lattice $\epsilon \mathbb{Z}^n$, with the capacity of each patch given by a stationary ergodic process. They found that, the minimizers $u_\epsilon$ to the energy functionals 
\[ J(u) = \int_D |\nabla u|^2 dx \]
converged to the energy minimizer of a special energy functional
\[ J_\alpha (u) = \int_D |\nabla u|^2 dx + \int_\Sigma \alpha (u - \varphi)_-^2 dS_x \]
which is the regular energy plus an extra penalty term along $\Sigma$, the flat portion of the boundary. Caffarelli and Mellet's result derived from their earlier work on homogenization of an obstacle in the interior of the domain \cite{Caffarelli-Mellet}, and used the framework laid down by Cioranescu and Murat \cite{Cioranescu-Murat}. Focardi has analyzed these problems in the context of $\Gamma$-convergence and in terms of  minimizing a nonlocal fractional order energy on the boundary surface \cite{focardi-1} \cite{focardi-2}. 

In this paper, we consider give conditions on the obstacle sets $S_\epsilon$ that are sufficient for there to be a homogenization limit, dispensing with the lattice and allowing them to live on a portion of the boundary that is locally a $C^{1,\alpha}$ graph. Specifically, we show that so long as a particular measure associated with $S_\epsilon$ converges to a surface measure in the appropriate space, the corresponding minimizers converge weakly to the minimizer of an energy with an added term. It is no longer required that each patch be centered on a lattice point (indeed, defining an appropriate lattice on a curved boundary is difficult), merely that each patch be well separated (a distance of about $\epsilon$) from every other patch, and be contained inside a small ball (specifically, of radius $M \epsilon^\frac{n-1}{n-2}$ for some $M$). 

Our method of proof tracks closely with the framework followed by Caffarelli and Mellet \cite{Caffarelli-Mellet-2}. In Section 2, we show that the homogenization result depends on the construction of an appropriate \textit{corrector}, which is a family of auxiliary functions that captures the effect of the obstacle set on the solution. In section 3, we construct this corrector for the special case of an obstacle $T_\epsilon$ consisting of balls centered on points of the boundary, and in section 4 we demonstrate that this corrector can be modified to the more general case of obstacles living on the boundary proper. 

\section{The Problem}
\subsection{Statement of the problem}
Consider an open bounded $C^{1,\alpha}$ domain $D \subset \mathbb{R}^n$, with $n\geq 3$, whose boundary is divided into two parts:
\[ \partial D = \Gamma \cup \Sigma\]
where $ \Gamma \cap \Sigma = \emptyset$. The first part, we call $\Gamma$. On the second portion, which we call $\Sigma$, there is a subset, known as the \textit{obstacle set}, $S_\epsilon \subset \Sigma$. 
We consider the minimizers of the Dirichlet energy
\[ \mathcal{J}(v) = \int_D |\nabla v|^2 dx \]
among functions $v(x)$ satisfying Dirichlet conditions on $\Gamma$ and a boundary obstacle problem on $S_\epsilon$; we call this set
the collection of admissible solutions to the obstacle problem:
\begin{enumerate}
 \item $v \in H^1(D)$
 \item $v|_\Gamma = \psi(x)$ for some smooth boundary data $\psi \in H^1(D)$
 \item $v|_{S_\epsilon} \geq \varphi$ for a smooth obstacle function $\varphi \in H^1(D)$
\end{enumerate}
There is a natural corresponding set of test functions, which satisfy
\begin{enumerate}
 \item $\phi \in C^\infty(\bar{D})$
 \item $\phi|_\Gamma = 0$
\end{enumerate}

The existence of a solution to the boundary obstacle problem is assured by the usual considerations of Hilbert space theory. 

To prove our main theorem, we will rely on a model obstacle set, which does not live on the boundary, but close to it. We will then be able to extend our proof from this model set to the general case stated in our problem. This set is constructed as follows:
\[ T_\epsilon = \left(\bigcup_{k} B_{r_{\epsilon,k}}(x_{\epsilon,k})\right) \bigcap D \]
where the $x_{\epsilon,k}$ are chosen to lie on $\Sigma$ at a distance of at least $2\epsilon$ from each other, and the radii of the balls are chosen so that there exist constants $c_1,c_2$, independent of $k, \epsilon$, so that
\[r_{\epsilon,k} = \tilde{r}_{\epsilon,k} \epsilon^{\frac{n-1}{n-2}}\] 
where $c_1  \leq \tilde{r}_{\epsilon,k} \leq c_2 $. Clearly, the number of such balls is $O(\epsilon^{1-n})$. To each obstacle set, we can assign a corresponding density function $\mu_{\epsilon}(x)$ living on $D$, which is associated with 
\[ \mu_{\epsilon}(x) = \sum_k (r_{\epsilon,k}^{n-2}) \frac{1}{\epsilon} \chi_{B_\epsilon(x_{\epsilon,k})} (x) \]

The model obstacle set can be generalized to arbitrary patches living on $\Sigma$, provided certain conditions hold. It is sufficient that each patch be a set of appropriately bounded \textit{capacity}, which is contained within $B_{M \epsilon^\frac{n-1}{n-2}} (x_{\epsilon,k}) \cap \Sigma$, where the $x_{\epsilon,k}$ are again points distributed on $\Sigma$, at a distance of at least $2 \epsilon$ from each other. We can call this set $S_{\epsilon,k}$ and its capacity is $\gamma_{\epsilon,k} \epsilon^{n-1}$. The overall obstacle is then
\[ S_\epsilon = \bigcup_k S_{\epsilon,k} \subset \Sigma \subset \partial D \]
The density corresponding to this is
\[ \mu_{\epsilon}(x) = \sum_k (\gamma_{\epsilon,k}) \frac{1}{\epsilon} \chi_{B_\epsilon(x_{\epsilon,k})} (x) \]
This definition in the case of balls is in accord with the one previously given, up to a constant which depends only on the spatial dimension. 

We now state our main theorem:
\begin{thm}
So long as $\mu_{\epsilon} dx \rightarrow \mu(x) dS_x $ in $H^{-1}$ as $\epsilon \rightarrow 0$, where $\mu(x)$ is some density on $\Sigma$, then the minimizers $u_\epsilon \rightharpoonup u_0$ in $H^1(D)$ where $u_0$ is the minimizer of the energy 
\[\mathcal{J}_\mu(v) = \int_D |\nabla v|^2 dx +  c_n \int_\Sigma (u-\varphi)_-^2 \mu(x) dS_x\]
 among all functions in $H^1(D)$ satisfying the Dirichlet conditions on $\Gamma$. The constant $c_n$ depends only on the spatial dimension $n$.  
\label{maintheorem}
\end{thm}
Here we abuse notation slightly to take $H^{-1}$ to be the dual space of $H$, defined as the subspace of $H^1(D)$ whose trace reduces to 0 on $\Gamma$. 

\subsection{Reduce Problem to finding a good corrector}\label{CorrectorSection}
Following Caffarelli and Mellet, we introduce the notion of a \textit{corrector} function, and demonstrate that the existence of a corrector meeting the appropriate conditions is sufficient to prove Theorem \ref{maintheorem}. The proofs of this subsection are essentially unchanged from \cite{Caffarelli-Mellet-2}. 
\begin{lem}\label{correctorlemma}
Assume that the obstacle set $S_\epsilon$ on the boundary is as defined. Then for each $\epsilon$ there exists a function $w_\epsilon(x) \in H^1(D)$ such that the following conditions are satisfied
\begin{equation}
 w_\epsilon(x) = 1 \textrm{ if } x \in S_\epsilon
\label{ccon1}
\end{equation}
\begin{equation}\label{ccon2}
\|w_\epsilon\|_{L^\infty(D)} \leq C \textrm{ where $C$ is bounded independent of $\epsilon$}
\end{equation}
\begin{equation}\label{ccon3}
 w_\epsilon \rightharpoonup 0 \textrm{ in } H^1(D)
\end{equation}
and for every sequence $v_\epsilon(x)$ satisfying 
\begin{eqnarray*}
v_\epsilon & \geq 0 & \textrm{ for } x \in S_\epsilon \\
\|v_\epsilon\|_{L^\infty(D)} &\leq C &\\
v_\epsilon &\rightharpoonup v & \textrm{ in } H^1(D) 
\end{eqnarray*}
and for any test function $\phi \in C^\infty(D)$ whose support is bounded away from $\Gamma$ we have
\begin{equation}\label{ccon4}
\lim_{\epsilon \rightarrow 0} \int_D \nabla w_\epsilon \cdot \nabla v_\epsilon \phi dx \geq - \int_\Sigma v \phi \mu(x) dS_x
\end{equation}
with equality if $v_\epsilon = 0$ on $S_\epsilon$. 
\label{mainlemma}
\end{lem}
The functions $\{w_\epsilon\}$ are called the corrector, and demonstrating the existence of a corrector satisfying those conditions is the subject of most of the paper. Two auxiliary lemmas are needed before we prove that the existence of such a corrector is sufficient to prove our theorem. 

\begin{lem}
For any test function $\phi \in \{ \phi \in C^\infty(\bar{D}) : \phi|_\Gamma = 0\}$ we have 
\[ \lim_{\epsilon \rightarrow 0} \int_D |\nabla w_\epsilon|^2 \phi  dx = \int_\Sigma \phi \mu(x) dS_x \]
\label{AuxLem1}
\end{lem}
\begin{proof}[Proof of Lemma \ref{AuxLem1}]
Let $v_\epsilon = 1-w_\epsilon$ where $w_\epsilon$ is the corrector. Then $v_\epsilon = 0$ on $T_\epsilon$, $v_\epsilon \rightharpoonup 1$ weakly in $H^1(D)$, and thus the result follows trivially from Lemma \ref{correctorlemma}. 
\end{proof}

\begin{lem}
Let $u_\epsilon$ be a (sub)sequence of minimizers in $H^1(D)$. If $u_\epsilon \rightharpoonup \bar{u}$ weakly in $H^1(D)$, then 
\[ \liminf_{\epsilon \rightarrow 0} \int_D |\nabla u_\epsilon|^2 dx \geq \int_D |\nabla \bar{u}|^2 dx + \int_\Sigma (\bar{u} - \varphi)_-^2 \mu(x) dS_x \]
\label{AuxLem2}
\end{lem}
\begin{proof}[Proof of Lemma \ref{AuxLem2}]
For an arbitrary function $v \in H^1(D)$ satisfying the Dirichlet conditions on $\Gamma$, we consider the function $v+(v-\varphi)_{-} w_\epsilon$, which is admissible for the obstacle problem at scale $\epsilon$. We evaluate the quantity
\[ \int_D |\nabla u_\epsilon - \nabla (v+(v-\varphi)_- w_\epsilon)|^2 dx \geq 0 \]
If we take the lim inf as $\epsilon \rightarrow 0$, we can take advantage of the fact that $w_\epsilon \rightharpoonup 0$ in $H^1(D)$, whence we can write
\begin{eqnarray*}
 \liminf \int_D |\nabla u_\epsilon|^2 + |\nabla v|^2 + |\nabla w_\epsilon|^2 (v-\varphi)_-^2 \\
- 2 \nabla u_\epsilon \cdot \nabla v - 2 (v-\varphi)_- \nabla u_\epsilon \cdot \nabla w_\epsilon dx \geq 0 
\end{eqnarray*}
We can apply Lemma \ref{AuxLem1} on the term in $|\nabla w_\epsilon|^2$ to get 
\[ \int |\nabla w_\epsilon|^2 (v-\varphi)_-^2 dx \rightarrow \int_\Sigma (v-\varphi)_-^2 \mu(x) dS_x \] 
The term in $\nabla u_\epsilon \cdot \nabla w_\epsilon$ can be rewritten through inequality (\ref{ccon4}), giving us
\[\liminf \int \nabla u_\epsilon \cdot \nabla w_\epsilon (v-\varphi)_-dx \geq -\int_\Sigma (\bar{u}-\varphi) (v-\varphi)_- \mu(x) dS_x. \]
We are justified in invoking the lemma as $u_\epsilon \in L^\infty(D)$ (owing to the maximum principle and the boundedness of both the boundary data $\psi$ and the obstacle function $\varphi$), and because as the solution to an obstacle problem $u_\epsilon - \varphi \geq 0$ on the obstacle set $S_\epsilon$.

We thus have
\begin{eqnarray*} 
\liminf \int_D |\nabla u_\epsilon|^2 dx &\geq& - \int_D |\nabla v|^2 + 2 \nabla \bar{u} \cdot \nabla v dx \\
&& - \int_\Sigma \left((v-\varphi)_-^2 + 2(\bar{u}-\varphi)(v-\varphi)_-\right) \mu(x) dS_x  
\end{eqnarray*}
If we let $v \rightarrow \bar{u}$ strongly in $H^1(D)$ (and hence strongly in $L^2$), then we can use the fact that $(\bar{u}-\varphi)(\bar{u}-\varphi)_- = -(\bar{u}-\varphi)_-^2$, getting the final result that
\[ \liminf_{\epsilon \rightarrow 0} \int_D |\nabla u_\epsilon|^2 dx \geq \int_D |\nabla \bar{u}|^2 dx + \int_\Sigma (\bar{u}-\varphi)_-^2 \mu(x) dS_x \]
\end{proof}

We now prove that the existence of a corrector is sufficient to demonstrate our theorem. 
\noindent
\begin{proof}[Proof of Theorem \ref{maintheorem}]
Consider for the arbitrary test function $v$, the function $v+(v-\varphi)_{-} w_\epsilon$, which is admissible as a possible solution for the obstacle problem at scale $\epsilon$. Applying Lemma \ref{AuxLem1} and expanding, we find that 
\[ \lim_{\epsilon \rightarrow 0} \mathcal{J}(v+(v-\varphi)_{-} w_\epsilon) = \mathcal{J}_\mu (v) \]
However, we also know that 
\[ \mathcal{J}(v+(v-\varphi)_{-} w_\epsilon) \geq \mathcal{J}(u_\epsilon) \]
since $u_\epsilon$ is the energy minimizer of all admissible functions at scale $\epsilon$. Thus we have
\[ \limsup \mathcal{J}(u_\epsilon) \leq \mathcal{J}_\mu(v) \]
for any test function $v$. However, Lemma \ref{AuxLem2} tells us that
\[ \liminf \mathcal{J}(u_\epsilon) \geq \mathcal{J}_\mu(\bar{u}) \] 
where $\bar{u}$ can be any subsequential limit of the $u_\epsilon$. Thus, 
\[ \mathcal{J}_\mu(\bar{u}) \leq \mathcal{J}_\mu(v) \] 
for arbitrary test functions $v$, which is to say any subsequential limit of the $u_\epsilon$ is the energy minimizer of $\mathcal{J}_\mu$. Thus, the $u_\epsilon$ converge to the energy minimizer of $\mathcal{J}_\mu$. 
\end{proof}

\section{Construction of a corrector}
We will first construct our corrector for the case of the model obstacle set $T_\varepsilon$ consists of arbitrarily placed balls of radius $r_{\varepsilon,k} = \tilde{r}_{\epsilon,k} \varepsilon^{\frac{n-1}{n-2}}$ centered on points $x_{\varepsilon,k} \in \Sigma$, where the $x_{\epsilon,k}$ are have distance at least $2\epsilon$ from each other. The corrector for $S_\epsilon$ will be an adaption of the corrector for $T_\epsilon$. The construction of the corrector borrows directly from \cite{Cioranescu-Murat}, and it is nothing more than a truncation of the suitably rescaled fundamental solution for the Laplacian. Let
\begin{eqnarray*}
w_{\varepsilon, k}(x) &=& 1 \textrm{ for } |x-x_{\varepsilon,k}| \leq r_{\varepsilon,k} \\
&=&  \frac{1}{ \frac{1}{r_{\varepsilon,k}^{n-2}} - 
\frac{1}{\varepsilon^{n-2}}} \left( \frac{1}{|x-x_{\varepsilon,k}|^{n-2}} - 
\frac{1}{\varepsilon^{n-2}} \right) \textrm{ for } r_{\varepsilon,k} < |x-x_{\epsilon,k}| < \varepsilon \\
&=&  0 \textrm{ for } \varepsilon \leq |x-x_{\epsilon,k}|
\end{eqnarray*}
Each of the $w_{\varepsilon,k}$ is supported on a ball of radius $\epsilon$ about some point $x_{\epsilon,k}$, and by assumption these balls are disjoint, so we simply take
\[ w_\epsilon = \sum_k w_{\epsilon,k} \]
as our corrector. 

\subsection{Simple properties of the corrector} 
That the corrector as constructed satisfies (\ref{ccon1}) and (\ref{ccon2}) is obvious. To demonstrate weak convergence in $H^1(D)$, or the condition in  (\ref{ccon3}), we show that $\|w_\epsilon\|_{L^2} \rightarrow 0$, and that $\|\nabla w_\epsilon\|_{L^2} \leq C$ where $C$ is independent of $\epsilon$. These are two reasonably straightforward calculations. 

\begin{lem}\label{L2bound}
\[ \lim_{\epsilon \rightarrow 0} \|w_\epsilon\|_{L^2(D)}^2 = 0 \]
\end{lem}
\begin{proof}[Proof of Lemma \ref{L2bound}]
We first perform the calculation inside a single ball of radius $\epsilon$. Without loss of generality, we assume that this ball is centered around 0, and we recall that $r_\epsilon = O(\epsilon^\frac{n-1}{n-2})$.   
\begin{eqnarray*}
 \int_{B_\epsilon} |w_\epsilon|^2 dx &=& |B_{r_\epsilon}| + \int_{B_\epsilon \setminus B_{r_\epsilon}} |w_\epsilon|^2 dx \\
&=& |B_{r_\epsilon}| + \int_{B_\epsilon \setminus B_{r_\epsilon}} \left(\frac{|x|^{2-n} - \epsilon^{2-n}}{r_\epsilon^{2-n} - \epsilon^{2-n}} \right)^2 dx \\
&=& |B_{r_\epsilon}| +\frac{ n \alpha(n)}{(r_\epsilon^{2-n}-\epsilon^{2-n})^2} \int_{r_\epsilon}^\epsilon \left(|x|^{4-2n} - 2 |x|^{2-n}\epsilon^{2-n} + \epsilon^{4-2n}\right)|x|^{n-1} d|x| \\
&=& |B_{r_\epsilon}| +\frac{ n \alpha(n)}{(r_\epsilon^{2-n}-\epsilon^{2-n})^2} \left( O(\epsilon^{4-n}) \right. \\
& & + \left. O(\epsilon^\frac{(4-n)(n-1)}{n-2}) + O(\epsilon^{\frac{2(n-1)}{n-2} + 2 - n}) + O(\epsilon^{\frac{n(n-1)}{n-2} + 4 - 2n}) \right) \\
&=& O(\epsilon^{\frac{n(n-1)}{n-2}}) + C O(\epsilon^{2n-2})\left( O(\epsilon^{4-n}) \right. \\
& & + \left. O(\epsilon^\frac{(4-n)(n-1)}{n-2}) + O(\epsilon^{\frac{2(n-1)}{n-2} + 2 - n}) + O(\epsilon^{\frac{n(n-1)}{n-2} + 4 - 2n}) \right) \\
&=& o(\epsilon^{n-1})
\end{eqnarray*}
Since there are on the order of $\epsilon^{1-n}$ such balls, our result holds. Note that we proved the result assuming the corrector is defined in the full ball. This is, of course, not the case, as the corrector is only defined in the intersection of $D$ with the ball. However, the bound on the full ball is a good upper bound. 
\end{proof}

\begin{lem}\label{H1bound}
\[  \|\nabla w_\epsilon\|_{L^2(D)}^2 \leq C \]
\end{lem}
\begin{proof}[Proof of Lemma \ref{H1bound}]
Again, we perform the calculation inside a single ball, as we did for Lemma \ref{L2bound}. 
\begin{eqnarray*}
\int_{B_\epsilon} |\nabla w_\epsilon|^2 dx &=& \int_{B_\epsilon \setminus B_{r_\epsilon}} \left((2-n) \frac{|x|^{1-n}}{r_\epsilon^{2-n} - \epsilon^{2-n}}\right)^2 dx \\
&=& n\alpha(n) \left(\frac{2-n}{r_\epsilon^{2-n} -\epsilon^{2-n}}\right)^2 \int_{r_\epsilon}^\epsilon |x|^{1-n} d|x| \\
&=& n\alpha(n) \frac{2-n}{\left(r_\epsilon^{2-n} -\epsilon^{2-n}\right)^2} (\epsilon^{2-n} - r_\epsilon^{2-n}) \\
&=& n\alpha(n) \frac{2-n}{\epsilon^{2-n} - r_\epsilon^{2-n}} = O(\epsilon^{n-1})
\end{eqnarray*}
Thus, since there are $O(\epsilon^{1-n})$ such balls, we have that the gradients of $w_\epsilon$ are uniformly bounded in $L^2$. 
\end{proof}

\subsection{Limiting property of the corrector}
Finally, to demonstrate (\ref{ccon4}) we rely strongly on fact that $\partial D$ is locally the graph of a $C^{1,\alpha}$ function. In what follows, we assume that $\epsilon$ is sufficiently small.
\[
 \int_{D} \phi \nabla v_\epsilon \nabla w_\epsilon dx = \sum_k \int_{(B_\epsilon(x_{\epsilon,k}) \setminus B_{r_{k,\epsilon}}(x_{\epsilon,k}))\cap D} \phi \nabla v_\epsilon \nabla w_\epsilon dx
\]
so we can examine the support of $w_\epsilon$ about each of the $x_{\epsilon, k}$ individually. Integrating by parts, we have that
\[ \int_{(B_\epsilon \setminus B_{r_\epsilon}) \cap D} \phi \nabla v_\epsilon \nabla w_\epsilon dx = \int_{\partial((B_\epsilon \setminus B_{r_\epsilon}) \cap D)} \phi v_\epsilon \partial_\nu w_\epsilon dS_x \]
since $\Delta w_\epsilon = 0$ in the annulus. We separate the three portions of the boundary: 
\[
\int_{\partial((B_\epsilon \setminus B_{r_\epsilon}) \cup D)} = \int_{\partial B_{r_\epsilon} \cap D} + \int_{\partial D \cap (B_\epsilon \setminus B_{r_\epsilon})} + \int_{\partial B_\epsilon \cap D}
\]

The first integral is purely positive, so it can be neglected for purposes of proving the inequality. It is easy to see that it is 0 when $v_\epsilon = 0$ on $T_\epsilon$, so it can also be neglected for proving the equality in that case. We will show the second integral goes to 0, and the third integral retrieves the desired term along the boundary in the limit. 
\subsubsection{The Second Term}\label{s-secondterm}
\begin{lem}
 The effect of the second term,
\[ \sum_k \int_{\partial D \cap (B_\epsilon(x_{\epsilon,k}) \setminus B_{r_{\epsilon,k}(x_{\epsilon,k})}} \phi v_\epsilon \partial_\nu w_\epsilon dS_x \rightarrow 0\]
as $\epsilon \rightarrow 0$. 
\end{lem}
\begin{proof}
We examine the second integral in a single ball. Here, we consider
\[ \int_{\partial D \cap (B_\epsilon \setminus B_{r_\epsilon})} \phi v_\epsilon \partial_\nu w_\epsilon dS_x \]
The idea is that $\partial D$ is ``almost flat'' for sufficiently small $\epsilon$, and that $w_\epsilon$ being a purely radial function, its derivative in a transverse direction is 0, so its derivative in an ``almost transverse'' direction is small.

Formally, consider a system of coordinates about $x_{\epsilon,k}$, which for ease of notation we take to be 0, and where $e_n$ represents the normal vector to the surface $\partial D$ at 0. The surface $\partial D$ can be represented as $x_n = \Sigma(x')$, where $x'$ are the first $n-1$ coordinates. Then $\Sigma(x')$ is a $C^{1,\alpha}$ function with tangent plane 0 at 0; so specifically we have the estimate
\[ |\Sigma(x')| \leq C |x'|^{1+\alpha}.\] 
Let $\nu(x')$ represent the normal vector at the point $(x', \Sigma(x')$. This is a $C^\alpha$ function. Now we can recast the integral as
\begin{equation} \int_{B_\epsilon \setminus B_{r_\epsilon}} \phi v_\epsilon \nabla w_\epsilon(x', \Sigma(x')) \cdot \nu(x') \sqrt{1 + |\nabla \Sigma(x')|^2} dx' 
\label{secondterm}
\end{equation}

Now we can write
\begin{eqnarray*}
 \nabla w_\epsilon(x',\Sigma(x')) \cdot \nu(x') &=& \nabla w_\epsilon(x',0) \cdot \nu(0) \\
& &+ (\nabla w_\epsilon(x',\Sigma(x')) - \nabla w_\epsilon(x',0)) \cdot \nu(0) \\
& &+ \nabla w_\epsilon(x',\Gamma(x')) \cdot( \nu(x') - \nu(0)) 
\end{eqnarray*}
The leading term here is 0 by construction. We will bound the second term by the $C^{1,\alpha}$ property of $\Sigma$, and the third
term with the $C^\alpha$ property of $\nu$. We first obtain a pointwise bound on the components of $\nabla w_\epsilon(x',0) - \nabla w_\epsilon(x',\Sigma(x'))$ which are not in the $x_n$ direction: 
\begin{eqnarray*}
 |\frac{\partial w_\epsilon}{\partial x_i} (x',\Sigma(x')) -  \frac{\partial w_\epsilon}{\partial x_i} (x',0)| &\leq& \sup_{0 \leq x_n \leq \Sigma(x')} |\frac{\partial^2 w_\epsilon}{\partial x_i \partial x_n}(x', x_n)| |\Sigma(x')|\\
&\leq & \frac{1}{\frac{1}{r_\epsilon^{n-2}} - \frac{1}{\epsilon^{n-2}}} \sup_{0 \leq x_n \leq \Sigma(x')} \frac{n(n-2) |x_i x_n|}{|x|^{n+2}} \Sigma(x') \\
&\leq& C \frac{n(n-2)}{r_\epsilon^{2-n} - \epsilon^{2-n}} \frac{|x_i|}{|x'|^{n+2}} |x'|^{2+2\alpha} \\
&\leq& C \frac{n(n-2)}{r_\epsilon^{2-n} - \epsilon^{2-n}} |x'|^{1+2\alpha-n}
\end{eqnarray*}
The bound on the derivative in the $x_n$ direction is similar, but includes an extra term:
\begin{eqnarray*}
 |\frac{\partial w_\epsilon}{\partial x_n} (x',\Sigma(x')) -  \frac{\partial w_\epsilon}{\partial x_n}(x',0)| &\leq& \sup_{0 \leq x_n \leq \Sigma(x')} |\frac{\partial^2 w_\epsilon}{\partial x_n^2}(x', x_n)| |\Sigma(x')|\\
&\leq & \frac{1}{\frac{1}{r_\epsilon^{n-2}} - \frac{1}{\epsilon^{n-2}}} \sup_{0 \leq x_n \leq \Sigma(x')} \left|\frac{n(n-2) x_n^2 }{|x|^{n+2}} + \frac{2-n}{|x|^n} \right| |\Sigma(x')| \\
&\leq& C \frac{n-2}{r_\epsilon^{2-n} - \epsilon^{2-n}} \left( \left|n|x'|^{1+3\alpha -n} + |x'|^{1+\alpha -n} \right|\right) 
\end{eqnarray*}
We recall that the regularity of the boundary lets us pick $\epsilon$ sufficiently small so that $|\nabla \Sigma| \leq 1$ for $|x'| < \epsilon$, and also that
\[ |\nu(x')-\nu(0)| \leq C|x'|^\alpha. \]
Thus, equation \ref{secondterm} can be directly bounded:

\begin{eqnarray*}
&&\int \phi v_\epsilon \nabla w_\epsilon(x', \Sigma(x')) \cdot \nu(x') \sqrt{1 + |\nabla \Sigma(x')|^2} dx' \leq \\
&& 2\|\phi v_\epsilon|_{L^\infty} \int \nabla w_\epsilon(x',0) \cdot \nu(0) + (\nabla w_\epsilon(x',\Sigma(x')) - \nabla w_\epsilon(x',0)) \cdot \nu(0) \\
&& + \nabla w_\epsilon(x',\Sigma(x')) \cdot( \nu(x') - \nu(0)) dx' \\
&&\leq 2 \|\phi v_\epsilon\|_{L^\infty} \int |\nabla w_\epsilon(x',\Sigma(x')) - \nabla w_\epsilon(x',0)| + |\nabla w_\epsilon(x',\Sigma(x'))| (C |x'|^\alpha) dx' \\
&&\leq \|\phi v_\epsilon\|_{L^\infty} \frac{C}{r_\epsilon^{2-n} - \epsilon^{2-n}} \int |x'|^{1+2\alpha -n} + |x'|^{1+3\alpha -n} + |x'|^{1+\alpha -n} dx' \\
&&\leq \|\phi v_\epsilon\|_{L^\infty} \frac{C}{r_\epsilon^{2-n} - \epsilon^{2-n}} \int_{r_\epsilon}^\epsilon \rho^{1+\alpha -n} \left(\rho^{\alpha} + \rho^{2\alpha} + 1\right) \rho^{n-2} d\rho \\
&&\leq \|\phi v_\epsilon\|_{L^\infty} \frac{C}{r_\epsilon^{2-n} - \epsilon^{2-n}} \epsilon^{\alpha} \\
&&\leq C\|\phi v_\epsilon\|_{L^\infty}  \frac{\epsilon^{n+\alpha-1}}{1-C\epsilon}
\end{eqnarray*}

Since there are on the order of $\epsilon^{1-n}$ such balls $k$ on the surface, the total integral from this term is $O(\epsilon^{\alpha})$, which goes to 0 as $\epsilon \rightarrow 0$. 
\end{proof}

\subsubsection{The third term}\label{s-thirdterm}
With only one term remaining, we see that to prove Lemma \ref{mainlemma} it suffices to show
\begin{lem}\label{thirdtermlemma}
 As $\epsilon \rightarrow 0$, we have 
\begin{equation}\label{thirdterm}
\lim_{\epsilon \rightarrow 0} \sum_k  \int_{\partial B_\epsilon(x_{\epsilon,k}) \cap D}  \phi v_\epsilon \partial_\nu w_\epsilon dS_x \geq -\int_\Sigma \phi v \mu(x) dS_x
\end{equation}
with equality if $v_\epsilon = 0$ on the $T_\epsilon$. 
\end{lem}

\begin{proof}
As before, we start by examining the behavior confined to a single ball, which for ease of notation we assume to be centered at 0. Notice that along $\partial B_\epsilon$, we have
\[ \partial_\nu w_\epsilon = \partial_r w_\epsilon = \frac{1-n}{r_\epsilon^{2-n} - \epsilon^{2-n}} |x|^{1-n} = -\tilde{r}_\epsilon^{n-2} \frac{n-1}{1 - \epsilon\tilde{r}_\epsilon^{n-2}}. \]
which is thus constant over the surface of the ball.

This inspires us ( following Cioranescu and Murat \cite{Cioranescu-Murat} ) to introduce the auxiliary function 
\begin{equation*} 
 q_\epsilon(x) = \begin{array}{lcr}
                  \left(\frac{1-n}{\tilde{r}_\epsilon^{2-n} - \epsilon}\right) \frac{1}{2\epsilon} (|x|^2 - \epsilon^2) &\textrm{for}& |x| \leq \epsilon \\
		  0 &\textrm{ for }& |x| > \epsilon
                 \end{array}
\end{equation*}
Thus,
\[\Delta q_\epsilon = \frac{1-n}{\tilde{r}_\epsilon^{2-n} - \epsilon} \frac{n}{\epsilon} \chi_{B_\epsilon}(x)\]
while the normal derivative along the boundary of $\partial B_\epsilon$ matches that of the corrector $w_\epsilon$. 

We see that $\|q_\epsilon\|_{L^\infty} \leq C \epsilon$, and similarly that 
\[ \|\nabla q_\epsilon\|_{L^p(B_\epsilon)}^p \leq C \epsilon^n \]
for every $1 \leq p < \infty$. If we repeat the construction of $q_\epsilon$ over all $O(\epsilon^{1-n})$ such balls, then we have
\[ \|\nabla q_\epsilon\|_{L^p(D)}^p \leq C \epsilon \]
which goes to 0 as $\epsilon \rightarrow 0$. In particular, $q_\epsilon \rightarrow 0$ in $H^1(D)$. Integrating over all balls, we have
\[ \int_{D} \phi v_\epsilon \Delta q_\epsilon dx + \int_D \nabla (\phi v_\epsilon) \nabla q_\epsilon dx = \sum_k \int_{\partial(B_\epsilon(x_k) \cap D)} \phi v_\epsilon \partial_\nu q_\epsilon dS_x \]
The second term on the left hand side is of order $\epsilon$ and so can be neglected, while the term on the right is the sum of the term in (\ref{thirdterm}), and another error: along $\partial D \cap B_\epsilon$, we have $\nabla q_\epsilon = \frac{1-n}{\tilde{r}_\epsilon^{2-n} - \epsilon} \frac{1}{\epsilon} x$. However, $\partial D$ is a $C^{1,\alpha}$ surface - using local coordinates and letting the tangent plane at 0 be $\{x_n=0\}$, we see that $|x_n| \leq C |x'|^{1+\alpha}$, while the normal vector to $\partial D$ is a $C^\alpha$ function, with $|\nu(x') - \nu(0)| \leq C |x'|^\alpha$, where $x'$ represents the other $n-1$ coordinates to $x$, and $\nu(0) = \hat{x_n}$ by definition. Thus
\begin{eqnarray*}
 |\nabla q_\epsilon \cdot \nu| &=& |\frac{1-n}{\tilde{r}_\epsilon^{2-n} - \epsilon} \frac{1}{\epsilon} x \cdot \nu| \\
&\leq& |\frac{1-n}{\tilde{r}_\epsilon^{2-n} - \epsilon} \frac{1}{\epsilon}| C |x'|^{1+\alpha} \\
&\leq& |\frac{1-n}{\tilde{r}_\epsilon^{2-n} - \epsilon} | C |\epsilon|^{\alpha}
\end{eqnarray*}
Integrating this over a surface of area $O(\epsilon^{n-1})$, where we have $O(\epsilon^{1-n})$ such surfaces, gives us a contribution of $O(\epsilon^\alpha)$, which goes to 0 as $\epsilon$ vanishes; in other words, the integral along $\partial D$ is, up to a small error, negligible, since $\partial D$ does not differ much from the tangent plane to which $\nabla q_\epsilon$ is orthogonal. 

Hence, the lemma simplifies to evaluating the effect of 
\[\sum_k \frac{1-n}{\tilde{r}_{\epsilon,k}^{2-n} - \epsilon} \int_{B_\epsilon(x_k) \cap D} \phi v_\epsilon \frac{n}{\epsilon} dx \]
But we see that $\frac{1}{\tilde{r}_{\epsilon,k}^{2-n} - \epsilon} = \tilde{r}_{\epsilon,k}^{n-2} \frac{1}{1 - \epsilon \tilde{r}_{\epsilon,k}^{n-2}}$, so it can be bounded from above by $\tilde{r}_{\epsilon,k}^{n-2} \frac{1}{1-C\epsilon}$, and from below by $\tilde{r}_{\epsilon,k}^{n-2}$. Thus we see that we have reduced ourselves to evaluating the effect of 
\[- \int_D \phi v_\epsilon \mu_\epsilon(x) dx \]
which by hypothesis converges to 
\[ - \int_\Sigma \phi v \mu(x) dS_x \]
as desired. 
\end{proof}
\section{Adapting the corrector to randomly shaped obstacles}
\subsection{The adapted corrector}
In this section, our goal will be to demonstrate that we can still fulfill the corrector conditions given in \S \ref{CorrectorSection} for obstacle sets that are not composed wholly of small balls. To be precise, we replace $B_{r_{\epsilon,k}}(x_k)$ with an arbitrary set $S_{\epsilon,k} \subset B_{M\epsilon^\frac{n-1}{n-2}}(x_{\epsilon,k}) \cap \Sigma$, so long as this set has a well-defined capacity equal to $\gamma(\epsilon,k) \epsilon^{n-1}$, where $\gamma$ is bounded from above. The capacity is defined with respect to $\mathbb{R}^n$.    
\[ \gamma(\epsilon,k) \epsilon^{n-1}= \sup_{v \in H^1(\mathbb{R}^n), v|_{S_{\epsilon,k}} \geq 1} \int_{\mathbb{R}^n} |\nabla v|^2 dx \]

The method of this section, which we borrow from \cite{Caffarelli-Mellet}, is to build a corrector by ``stitching'' together a translation of the \textit{capacitary potential} of $S_{\epsilon,k}$ inside $B_{M\epsilon^\frac{n-1}{n-2}}(x_k)$ with the corrector we have already constructed for some ``equivalent'' radius $r_{\epsilon,k}$ up to $B_\epsilon(x_k)$. 

Define $r_{\epsilon,k}$ to be the radius of the ball with capacity $\gamma(\epsilon,k) \epsilon^{n-1}$. Since the capacity of a ball of radius $r$ in $\mathbb{R}^n$ is $c_n r^{n-2}$ where $c_n$ depends only on dimension, it is clear that $\tilde{r}_{\epsilon,k} = r_{\epsilon,k} \epsilon^{-\frac{n-1}{n-2}}$ is a bounded quantity, with $\gamma(\epsilon,k) = c_n \tilde{r}_{\epsilon,k}^{n-2}$. 

We make use of a lemma, proved in \cite{Caffarelli-Mellet} as Lemma 5.3, which says that for any set $S$ confined to a ball of radius $M$ with capacity $\gamma$, if $\psi$ is the capacitary potential of $S$ and $N$ the fundamental solution centered at 0, then 
\[ |\psi(x) - \gamma N(x)| \leq \frac{C_M}{|x|} N(x) \]
for all $|x| \geq \frac{1}{2M}$. We reproduce the proof in an appendix for the reader's convenience. For our needs, we consider the behavior of this inequality at the ``intermediate scale'' represented by $a_\epsilon = \epsilon^\frac{n-\frac{3}{2}}{n-2}$, to be precise, for $\frac{1}{2} a_\epsilon \leq |x| \leq 4 a_\epsilon$, making the substitutions $M \rightarrow M \epsilon^\frac{n-1}{n-2}$, and $\gamma \rightarrow \gamma(k,\epsilon) \epsilon^{n-1}$: 
\[ |\psi(x) - \epsilon^{n-1} \gamma(\epsilon,k) N(x)| \leq C \epsilon^{\frac{1}{2}(1+\frac{1}{n-2})} \]
We assume $\epsilon$ is sufficiently small that $M \epsilon^\frac{n-1}{n-2} < \frac{a_\epsilon}{2}$. 

Consider now a single ball $B_\epsilon$, with a set $S_\epsilon$ representing the obstacle set, contained inside $B_{M\epsilon^\frac{n-1}{n-2}}$. Let $\psi_\epsilon(x)$ be the capacitary potential of $S_\epsilon$ with respect to $\mathbb{R}^n$, and let $w_\epsilon(x)$ be the standard corrector as defined in \S 2, using the equivalent radius $r_\epsilon$. Let $\eta_\epsilon(x)$ be the usual bridging function, which is smooth, takes the value of 1 inside $B_{a_\epsilon}$, 0 outside $B_{2a_\epsilon}$, with $|\nabla \eta_\epsilon| \leq \frac{C}{a_\epsilon}$ and $|\Delta \eta_\epsilon| \leq \frac{C}{a_\epsilon^2}$, and define our new corrector by 
\[\hat{w}_\epsilon = \eta_\epsilon(x) \psi_\epsilon(x) + (1-\eta_\epsilon(x)) w_\epsilon(x) \]
repeating the construction abou each $x_{\epsilon,k}$. 

\subsection{Corrector conditions}
We need to check that our new corrector satisfies the corrector conditions. That (\ref{ccon1}) and (\ref{ccon2}) are satisfied is easy, as is the $L^2$ portion of (\ref{ccon3}):
\begin{enumerate}
 \item That we have $\hat{w}_\epsilon = 1$ on $S_{\epsilon,k}$ is clear from the definition of $\psi_\epsilon$. 
 \item That $\hat{w}_\epsilon \rightarrow 0$ in $L^2(D)$ is also clear, since $w_\epsilon \rightarrow 0$ in $L^2(D)$, and $\hat{w}_\epsilon$ only differs from $w_\epsilon$ inside $\bigcup B_{2a_\epsilon}(x_k)$, which, summed over all balls, has volume $O(\epsilon^{1-n} \epsilon^{\frac{n(n-\frac{3}{2})}{n-2}}) = O(\epsilon^{1+\frac{n}{2(n-2)}})$, and as both functions are bounded by 1, this means the $L^2$ norm of their difference goes to 0. 
\end{enumerate}

To show that $\|\nabla \hat{w}_\epsilon \|_{L^2(D)}$ is uniformly bounded, and that (\ref{ccon4}) are satisfied, is a slightly more delicate matter. 

Let us begin by writing 
\begin{equation}\label{nablaexpansion}
  \nabla \hat{w}_\epsilon = \eta_\epsilon \nabla \psi_\epsilon + (1-\eta_\epsilon) \nabla w_\epsilon + (\psi_\epsilon - w_\epsilon) \nabla \eta_\epsilon 
\end{equation}
We can certainly write
\[ |\nabla \hat{w}_\epsilon|^2 \leq C \left( \eta^2 |\nabla \psi_\epsilon|^2 + (1-\eta)^2 |\nabla w_\epsilon|^2 + (\psi_\epsilon - w_\epsilon)^2 |\nabla \eta_\epsilon|^2 \right) \]
The first two terms present no difficulty for us, as $\psi_\epsilon$ has energy of order $O(\epsilon^{n-1})$, which summed over $O(\epsilon^{1-n})$ balls is $O(1)$. Similarly, we already know that the energy of $w_\epsilon$ is uniformly bounded by the arguments of the previous section. Thus the problem becomes one of estimating $|\psi_\epsilon - w_\epsilon|$ on the support of $\nabla \eta_\epsilon$, or in the annulus $B_{2a_\epsilon} \setminus B_{a_\epsilon}$, which we settle with the following estimate. 
\begin{lem}\label{L-infinity-estimate}
 Inside the annulus $B_{4a_\epsilon} \setminus B_{\frac{1}{2} a_\epsilon}$, we have that 
\[ |\psi_\epsilon - w_\epsilon | = O(\epsilon) \]
\end{lem}
\begin{proof}
Not surprisingly, we will be comparing $w_\epsilon$ with $\gamma(\epsilon) \epsilon^{n-1} N(x)$, and $\gamma(\epsilon) \epsilon^{n-1} N(x)$ with $\psi_\epsilon$. We claim first that on the annulus, 
\[ |w_\epsilon(x) - \gamma(\epsilon) \epsilon^{n-1} N(x) | = O(\epsilon) \]
This is done by noting that 
$$\gamma(\epsilon) \epsilon^{n-1} N(x) = \frac{r_\epsilon^{n-2}}{|x|^{n-2}}$$, and recalling that 
$$ w_\epsilon(x) = \frac{|x|^{2-n} - \epsilon^{2-n}}{r_\epsilon^{2-n} - \epsilon^{2-n}} $$
A bit of calculation gives us that 
\[ \frac{1}{r_\epsilon^{2-n} -\epsilon^{2-n}} - r_\epsilon^{n-2} = \epsilon^{n-1} \tilde{r}_\epsilon^{n-2} \left(1 - \frac{1}{1-\epsilon\tilde{r}_\epsilon^{n-2}}\right) = \epsilon^{n-1}\tilde{r}_\epsilon^{n-2} O(\epsilon) = O(\epsilon^n) \]
The constant portion of $w_\epsilon(x)$ is of order $\epsilon$ can be easily checked. Now taking advantage of the fact that for $a_\epsilon \leq |x| \leq 2 a_\epsilon$, we have $|x| = O(\epsilon^\frac{n-\frac{3}{2}}{n-2})$, we see that 
\[ |w_\epsilon - \gamma(\epsilon)\epsilon^{n-1} N(x)| = O(\epsilon^n) \frac{1}{|x|^{n-2}} + O(\epsilon) = O(\epsilon) \]
The next part, to compare $\gamma(\epsilon)\epsilon^{n-1} N(x)$ with $\psi_\epsilon(x)$, uses our lemma, which tells us that 
\[ |\gamma(\epsilon) \epsilon^{n-1} N(x) - \psi_\epsilon(x)| \leq C \epsilon^{1 + \frac{n}{2(n-2)}} = o(\epsilon) \]
on the scale with which we are concerned. Thus, 
\[ |w_\epsilon - \psi_\epsilon| = O(\epsilon) \]
\end{proof}
With this estimate in hand, it is not hard to see that 
\[ |\nabla \eta_\epsilon|^2 |w_\epsilon - \psi_\epsilon|^2 \leq C \frac{\epsilon^2}{a_\epsilon^2}.\]
Integrating over the annulus, which has volume $O(a_\epsilon^n) = O(\epsilon^{\frac{n(n-\frac{3}{2})}{n-2}})$, we find that
\[ \int_{B_{2a_\epsilon}\setminus B_{a_\epsilon}} (w_\epsilon - \psi_\epsilon)^2 |\nabla \eta_\epsilon|^2 dx = O(\epsilon^{n+\frac{1}{2}}) \]
Summing over the entire volume of $O(\epsilon^{1-n})$ balls, we see that the contribution of this term goes to 0 as $\epsilon \rightarrow 0$. 

An estimate related to Lemma \ref{L-infinity-estimate} we will use later is the following:
\begin{lem}\label{gradient-estimate}
 Inside the annulus $B_{2a_\epsilon} \setminus B_{a_\epsilon}$, we have that
\[ |\nabla \psi_\epsilon - \nabla w_\epsilon| = O(\epsilon^\frac{-1}{2(n-2)}) \]
\end{lem}
\begin{proof}
We notice that inside the annulus $B_{4a_\epsilon} \setminus B_{\frac{1}{2} a_\epsilon}$, both $\psi_\epsilon$ and $w_\epsilon$ are harmonic functions, and hence we can apply interior gradient estimates, the distance to the boundary being of $O(a_\epsilon)$.
\end{proof}

\subsection{Convergence of the integral term}
We are left only with showing that $\hat{w}$ satisfies the requirements of (\ref{ccon4}). We do this by arguing that the change in $\int_D \phi \nabla w_\epsilon \cdot \nabla v_\epsilon dx$ introduced by replacing $w_\epsilon$ by $\hat{w}_\epsilon$ is negligible. Integrating by parts on a single ball, we find that
\[\int_{B_\epsilon \cap D} \phi \nabla \hat{w}_\epsilon \cdot \nabla v_\epsilon dx = \int_{\partial(B_\epsilon \cap d)} \phi v_\epsilon \partial_\nu \hat{w}_\epsilon dS_x - \int_{B_\epsilon \cap D} v_\epsilon \nabla \phi \cdot \nabla \hat{w}_\epsilon + v_\epsilon \phi \Delta \hat{w}_\epsilon dx  \]
We will argue that, as $\epsilon \rightarrow 0$, both interior terms go to zero, and the boundary term splits into three parts: a purely positive component lying along $S_\epsilon$ which is 0 when $v_\epsilon = 0$ on $S$, a term lying on $\partial D$ which tends to 0, and a term that recovers the limit we are looking for. 

The first interior term can be estimated directly: since $\hat{w}_\epsilon \rightharpoonup 0$ in $H^1(D)$, and $v_\epsilon$ is bounded in $L^2(D)$, it is clear that that term goes to 0 as $\epsilon \rightarrow 0$. 

The second interior term relies on an estimate of $\Delta \hat{w}_\epsilon$. Inside $B_{a_\epsilon}$ and $B_\epsilon \setminus B_{2a_\epsilon}$, $\hat{w}_\epsilon$ is equal to $\psi_\epsilon$ and $w_\epsilon$, respectively, and these are both harmonic functions. Thus, we need only consider the behavior of the interior term in the annulus $B_{2 a_\epsilon} \setminus B_{a_\epsilon}$. Noting that 
\[ \Delta \hat{w_\epsilon} = \eta_\epsilon \Delta \psi_\epsilon + (1-\eta_\epsilon) \Delta w_\epsilon + \Delta \eta_\epsilon (\psi_\epsilon - w_\epsilon) + 2 \nabla \eta_\epsilon \cdot (\nabla \psi_\epsilon - \nabla w_\epsilon) \]
we proceed to bound each term. $\Delta \psi_\epsilon = \Delta w_\epsilon = 0$, and by construction we have $|\Delta \eta_\epsilon| \leq \frac{C}{a_\epsilon^2} = O(\epsilon^{-\frac{2(n-3/2)}{n-2}})$, $|\nabla \eta_\epsilon| \leq \frac{C}{a_\epsilon} = O(\epsilon^{-\frac{n-3/2}{n-2}})$. We invoke Lemmas \ref{L-infinity-estimate} and \ref{gradient-estimate} to control the terms in $\psi_\epsilon - w_\epsilon$, the result being that 
\[ |\Delta \hat{w_\epsilon}| \leq C \frac{\epsilon}{a_\epsilon^2} = C \epsilon^{\frac{1-n}{n-2}} \]
Thus 
\[ \int_{B_\epsilon \cap D } \phi v_\epsilon \Delta \hat{w}_\epsilon \leq C a_\epsilon^n \epsilon^{\frac{1-n}{n-2}} = O(\epsilon^{n-1/2}) \]
Repeating this integral over all balls, we find that the second interior term has contribution $O(\epsilon^{\frac{1}{2}})$, which goes to 0 as $\epsilon \rightarrow 0$. 

The boundary of $B_\epsilon \cap D$ has two components: 
\[ \partial(B_\epsilon \cap D) = (\partial D \cap B_\epsilon) \cup (\partial B_\epsilon \cap D).\] 
On $\partial B_\epsilon \cap D$, we have $\hat{w}_\epsilon = w_\epsilon$, and so we can apply the result from \S \ref{s-thirdterm} to show that it recovers our desired limit. Thus, we are only concerned with showing that we can neglect the contribution of the other portion of the boundary, specifically, that portion of it lying in $\partial D \cap B_\epsilon$. 

We seek to control 
\[ \int_{\partial D \cap B_\epsilon} \phi v_\epsilon \partial_\nu \hat{w}_\epsilon dS_x \] 
which is more profitably written as 
\[ \int_{\partial D \cap B_\epsilon} \phi v_\epsilon \left(\eta_\epsilon \partial_\nu \psi_\epsilon + (1-\eta_\epsilon)\partial_\nu w_\epsilon + (\psi_\epsilon - w_\epsilon) \partial_\nu \eta_\epsilon\right) dS_x \]
We consider each term individually. 

The arguments of \S \ref{s-secondterm} apply to $\int_{\partial D \cap B_\epsilon} (1-\eta_\epsilon)\partial_\nu w_\epsilon dS_x$, and suffice to show that term is 0. 

For the term $\int_{\partial D \cap B_\epsilon} (\psi_\epsilon - w_\epsilon) \partial_\nu \eta_\epsilon dS_X$, we apply Lemma \ref{L-infinity-estimate} to control $\psi_\epsilon - w_\epsilon$ on the annulus where $\nabla \eta_\epsilon$ is supported, as well as the estimate $|\nabla \eta_\epsilon| \leq \frac{C}{a_\epsilon}$, to get 
\[ \left| (\psi_\epsilon - w_\epsilon) \partial_\nu \eta_\epsilon \right| \leq C \epsilon^{1 - \frac{n-3/2}{n-2}} \]
Since $\nabla \eta_\epsilon$ only lives on the annulus $B_{2a_\epsilon} \setminus B_{a_\epsilon}$, we can use the bound that $|\partial D \cap B_{2 a_\epsilon}| \leq C a_\epsilon^{n-1}$ which gives us the basic estimate that
\[ \left|\int_{\partial D \cap B_\epsilon} (\psi_\epsilon - w_\epsilon) \partial_\nu \eta_\epsilon dS_x \right| \leq C \epsilon^{1 + n - 3/2} = o(\epsilon^{n-1}) \]
Summing over $O(\epsilon^{1-n})$ such balls, we see that the contribution of this term goes to 0. 

For the last term, we have to do a little more work: 
\begin{lem}
\[ \sum_k \int_{\partial D \cap B_\epsilon(x_k)} \phi v_\epsilon \eta_\epsilon \partial_\nu \psi_{\epsilon,k} dS_x \rightarrow 0 \]
as $\epsilon \rightarrow 0$. 
\end{lem}
\begin{proof}
As usual, we do our proof over a single ball, showing that the term is so small that, when summed over all balls, it still goes to 0. We consider the nature of $\psi_\epsilon$. It is well known that (see, for example, the classical paper of Littman, Stampacchia, and Weinberger \cite{LSW}, sections 5 and 6) the capacitary potential $\psi$ of a set $S$ solves the equation 
\[ \Delta \psi = \sigma \]
where $\sigma$ is a measure supported on the boundary of $S$, known as the \textit{capacitary distribution}, and $\sigma(S)$ is the capacity of $S$. Furthermore, $\psi$ can be thought of as the Newtonian potential of $\sigma$. For our purposes, if we let $\sigma_\epsilon$ be the capacitary distribution corresponding to the portion of the obstacle set living in the center of the ball $S_\epsilon$, this means that 
\[ \psi_\epsilon(x) = \int_{S_\epsilon} \frac{c_n}{|x-y|^{n-2}} d\sigma(y) \]
Let us analyze the normal derivative of $\frac{1}{|x-y|^{n-2}}$. Since $x$ and $y$ both lie on the $C^{1,\alpha}$ surface $\partial D$, $|\nu(x) -\nu(y)| \leq C |x-y|^\alpha$. We recall that 
\[ \nabla \frac{1}{|x-y|^{n-2}} = (2-n) \frac{x-y}{|x-y|^n} \]
Thus, $(2-n) \frac{x-y}{|x-y|^n} \cdot \nu(x)$ is the term to be estimated. To estimate the normal derivative, we set local coordinates centered at $x$, such that the tangent plane to $\partial D$ at $x$ is $\{x_n=0\}$, and $\partial D$ is locally represented by a $C^{1,\alpha}$ function $x_n =\Gamma(x')$, where $x'$ represents the first $n-1$ components. Thus, in local coordinates, we have that $|x_n - y_n| \leq C |x-y|^{1+\alpha}$. Interestingly, $\nu(x) = e_n$ in our coordinates, so 
\[ |\nabla \frac{1}{|x-y|^{n-2}} \cdot \nu(x)| = |(2-n) \frac{x_n - y_n} {|x-y|^n}| \leq C \frac{1}{|x-y|^{n-1-\alpha}} \]
Returning to our original integral, we find that, since $\eta_\epsilon$ is only supported on a ball of radius $2 a_\epsilon$, we have
\begin{align*}
\left|\int_{\partial D \cap B_\epsilon} \phi v_\epsilon \eta_\epsilon \partial_\nu \psi_\epsilon dS_x\right| &\leq \int_{\partial D \cap B_{2 a_\epsilon}} |\phi v_\epsilon \eta_\epsilon \partial_\nu \psi_\epsilon| dS_x \\
&\leq \|\phi v_\epsilon\|_{L^\infty} \int_{\partial D \cap B_{2 a_\epsilon}} |\partial_\nu \psi_\epsilon| dS_x \\
&\leq \|\phi v_\epsilon\|_{L^\infty} \int_{\partial D \cap B_{2 a_\epsilon}} \left| \partial_\nu \left(\int \frac{c_n}{|x-y|^{n-2}} d\sigma(y) \right)\right| dS_x \\
&\leq \|\phi v_\epsilon\|_{L^\infty} \int_{\partial D \cap B_{2 a_\epsilon}} \int \frac{C}{|x-y|^{n-1-\alpha}} d\sigma(y) dS_x \\
\intertext{From here, we use Fubini to change the order of integration, and integrate in polar coordinates w.r.t $x$, remembering that $\sigma$ is supported inside a radius much smaller than $a_\epsilon$}
&\leq C \|\phi v_\epsilon\|_{L^\infty} \int a_\epsilon^\alpha d\sigma(y) \\
&\leq C \|\phi v_\epsilon\|_{L^\infty} \gamma_\epsilon \epsilon^{n-1 + \alpha \frac{n-3/2}{n-2}}
\end{align*}
Since $v_\epsilon$ is uniformly in $L^\infty$, and $\gamma_\epsilon$ is bounded, we have that this term is $o(\epsilon^{n-1})$, which yields the desired result. 
\end{proof}

\section{Acknowledgements}
The author would like to thank his advisor, Luis Caffarelli, for posing the problem and for his unfailing patience and guidance. The author was supported by the (NSF-funded) Research Training Group in Applied and Computational Mathematics at the University of Texas at Austin (NSF Award No. 0636586) during the preparation of this paper.




\bibliography{skeleton}{}
\bibliographystyle{amsplain}
\end{document}